\documentclass[12pt]{amsart}
\usepackage[latin1]{inputenc}
\usepackage{amsbsy}
\usepackage{amscd} 
\usepackage{amsfonts}
\usepackage{amsgen} 
\usepackage{amsmath}
\usepackage{amsopn} 
\usepackage{amssymb}
\usepackage{amstext}
\usepackage{amsthm} 
\usepackage{amsxtra}
\usepackage[all]{xy}
\usepackage{booktabs}

\theoremstyle{plain} 
\newtheorem{thm}{Theorem}[section]
\newtheorem*{sat}{Main Theorem}

\newtheorem{lem}[thm]{Lemma}

\theoremstyle{definition}
\newtheorem{defn}[thm]{Definition}
\newtheorem{rem}[thm]{Remark}
\newtheorem{ex}[thm]{Example}

\numberwithin{equation}{section}

\renewcommand{\theta}{\vartheta}
\renewcommand{\phi}{\varphi}
\renewcommand{\epsilon}{\varepsilon}
\renewcommand{\subset}{\subseteq}

\newcommand{\N}{\mathbb N}
\newcommand{\Z}{\mathbb Z}

\newcommand{\C}{\mathbb C}

\DeclareMathOperator{\Aut}{Aut}

\newcommand{\QBan}{G_{aut}^+}
\newcommand{\QBic}{G_{aut}^*}

\begin{document}
\title{Quantum symmetries of graph $C^*$-algebras}
\author{Simon Schmidt}
\author{Moritz Weber}
\address{Saarland University, Fachbereich Mathematik, 
66041 Saarbr\"ucken, Germany}
\email{simon.schmidt@math.uni-sb.de, weber@math.uni-sb.de}
\thanks{The second author was partially funded by the ERC Advanced Grant NCDFP, held by Roland Speicher. This work was part of the first author's Master's thesis. This work was also supported by the DFG project \emph{Quantenautomorphismen von Graphen}}

\date{\today}

\begin{abstract}
The study of graph $C^*$-algebras has a long history in operator algebras. Surprisingly, their quantum symmetries have never been computed so far. We close this gap by proving that the quantum automorphism group of a finite, directed graph without multiple edges acts maximally on the corresponding graph $C^*$-algebra. This shows that the quantum symmetry of a graph coincides with the quantum symmetry of the graph $C^*$-algebra. In our result, we use the definition of quantum automorphism groups of graphs as given by Banica in 2005. Note that Bichon gave a different definition in 2003; our action is inspired from his work. We review and compare these two definitions and we give a complete table of quantum automorphism groups (with respect to either of the two definitions) for undirected graphs on four vertices. 
\end{abstract}

\maketitle

\section*{Introduction}

Symmetry constitutes one of the most important properties of a graph. It is captured by its automorphism group
\[\Aut(\Gamma):=\{\sigma\in S_n\;\mid\;\sigma\epsilon=\epsilon\sigma\}\subset S_n,\]
where $\Gamma=(V,E)$ is a finite graph with $n$ vertices and no multiple edges, $\epsilon\in M_n(\{0,1\})$ is its adjacency matrix, and $S_n$ is the symmetric group. In modern mathematics, notably in operator algebras, symmetries are no longer described only by groups, but by quantum groups. In 2005, Banica \cite{QBan} gave a definition of a quantum automorphism group of a finite graph  within  Woronowicz's theory of compact matrix quantum groups \cite{CMQG2}. In our notation,  $\QBan(\Gamma)$ is based on the $C^*$-algebra 
\allowdisplaybreaks
\begin{align*}
C(&\QBan(\Gamma)):= C(S_n^+) / \langle u\epsilon=\epsilon u\rangle\\
&=C^*(u_{ij}, i,j=1,\ldots,n\;\mid\; u_{ij}=u_{ij}^*=u_{ij}^2, \sum_l u_{il}=1=\sum_l u_{lj}, R_{\mathrm{Ban}} \text{ hold}), 
\end{align*}
where $S_n^+$ is Wang's quantum symmetric group \cite{WanSn} and $R_{\mathrm{Ban}}$ are the relations 
\begin{align*}
\sum_k u_{ik}\epsilon_{kj}=\sum_k \epsilon_{ik}u_{kj}.
\end{align*}
Earlier, in 2003, Bichon \cite{QBic} defined a quantum automorphism group $\QBic(\Gamma)$ via
\begin{align*}
C(&\QBic(\Gamma))\\
&:=C^*(u_{ij}, i,j=1,\ldots,n\;\mid\; u_{ij}=u_{ij}^*=u_{ij}^2, \sum_l u_{il}=1=\sum_l u_{lj}, R_{\mathrm{Bic}} \text{ hold}),
\end{align*}
where $R_{\mathrm{Bic}}$ are the relations
\begin{align*}
\sum_k u_{ik}\epsilon_{kj}=\sum_k \epsilon_{ik}u_{kj}, \quad u_{s(e)s(f)}u_{r(e)r(f)}=u_{r(e)r(f)}u_{s(e)s(f)} \textnormal{ for } e,f\in E,
\end{align*}
and $r:E\to V$ and $s:E\to V$ are range and source maps respectively.
We immediately see that
\[\Aut(\Gamma)\subset\QBic(\Gamma)\subset\QBan(\Gamma)\]
holds, in the sense that there are surjective $^*$-homomorphisms:
\begin{align*}
&C(\QBan(\Gamma))&&\to &&C(\QBic(\Gamma)) &&\to &&C(\Aut(\Gamma))\\
&u_{ij}&&\mapsto&&u_{ij}&&\mapsto&&(\sigma\mapsto \sigma_{ij})
\end{align*}
Relatively little is known about these two quantum automorphism groups of graphs and we refer to Section \ref{SectLit} for an overview on all published articles in this area.

Graph $C^*$-algebras in turn are well-established objects in operator algebras. They emerged from Cuntz and Krieger's work \cite{CK} in the 1980's and they developed to be one of the most important classes of examples of $C^*$-algebras, see for instance Raeburn's book for an overview \cite{Rae}. Given a finite graph $\Gamma=(V,E)$ the associated graph $C^*$-algebra $C^*(\Gamma)$ is defined as 
\begin{align*}
C^*(\Gamma):=C^*(p_v, v\in V, s_e, e\in E\;\mid\;& p_v=p_v^*=p_v^2, \;p_vp_w=0 \textnormal{ for } v\neq w,\\
 &s_e^*s_e=p_{r(e)}, \sum_{\substack{e\in E\\ s(e)=v}} s_es_e^*=p_v, \textnormal{ if } s^{-1}(v)\neq \emptyset).
\end{align*}

A natural question is then:  What is the quantum symmetry group of the graph $C^*$-algebra and is it one of the above two quantum automorphism groups of the underlying graphs? The answer is: It is given by the one  defined by Banica. Note however, that Bichon's definition has its justification in other contexts such as in \cite{GoBhSk,Bic2} or in the recent work by Speicher and the second author \cite{SW16}. Moreover, Bichon's work \cite{QBic} inspired us how to formulate our main theorem, see also Remark \ref{RemBicInsp}.

\pagebreak
\section{Main result}\label{SectMain}

Intuitively speaking, our main result is that the quantum symmetry of a finite graph without multiple edges coincides with the quantum symmetry of the associated graph $C^*$-algebra. In other words, the following diagram is commutative:
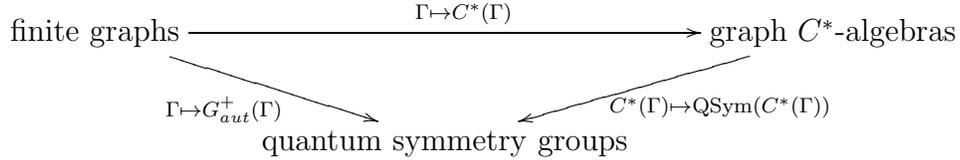
\begin{displaymath}
    \xymatrix{
        \textnormal{finite graphs} \ar[dr]_{\Gamma\mapsto\QBan(\Gamma)\;\;\;\;\;\;} \ar[rr]^{\Gamma\mapsto C^*(\Gamma)} &  & \textnormal{graph $C^*$-algebras} \ar[dl]^{\;\;\;\;\;\;\;\;\;\;\;\;C^*(\Gamma)\mapsto \textnormal{QSym}(C^*(\Gamma))}\\
                             & \textnormal{quantum symmetry groups} }
\end{displaymath}

More precisely, we have the following result.

\begin{sat}
Let $\Gamma$ be a finite graph with $n$ vertices $V=\{ 1, ..., n\}$ and $m$ edges $E = \{ e_1, ..., e_m\}$ having no multiple edges. The maps 
\begin{align*}
\alpha: C^*(\Gamma) &\to C(\QBan(\Gamma))\otimes C^*(\Gamma) ,\\
p_i &\mapsto \sum_{k = 1}^n u_{ik} \otimes p_k, &&1 \leq i \leq n,\\
s_{e_j} &\mapsto \sum_{l=1}^m  u_{s(e_j)s(e_l)}u_{r(e_j)r(e_l)} \otimes s_{e_l}, &&1 \leq j \leq m,
\intertext{and}
\beta: C^*(\Gamma) &\to  C(\QBan(\Gamma)) \otimes C^*(\Gamma) ,\\
p_i &\mapsto \sum_{k = 1}^n u_{ki} \otimes  p_k, &&1 \leq i \leq n,\\
s_{e_j} &\mapsto \sum_{l=1}^m  u_{s(e_l)s(e_j)}u_{r(e_l)r(e_j)}\otimes s_{e_l},  &&1 \leq j \leq m
\end{align*} 
define a left and a right action of $\QBan(\Gamma)$ on $C^*(\Gamma)$, respectively. Moreover, whenever $G$ is a compact matrix quantum group acting on $C^*(\Gamma)$ in the above way, we have $G \subset \QBan(\Gamma)$. In this sense, the quantum automorphism group $\QBan(\Gamma)$ of $\Gamma$ is the quantum symmetry group of $C^*(\Gamma)$, see also Remark \ref{RemQSym}.
\end{sat}

We also provide some tools for comparing and dealing with the two definitions of quantum automorphism groups of graphs, $\QBan(\Gamma)$ and $\QBic(\Gamma)$, notably depending on the complement $\Gamma^c$ of $\Gamma$, see Section \ref{SectComplement}. Moreover, we provide a list of all $\Aut(\Gamma)$, $\QBan(\Gamma)$ and $\QBic(\Gamma)$ for undirected graphs $\Gamma$ on four vertices, having no multiple edges and no loops, see Section \ref{SectTabelle}.

\pagebreak
\section{Preliminaries}

\subsection{Graphs}

We fix some notations for graphs used throughout this article. A graph $\Gamma=(V,E)$ is \emph{finite}, if the set $V$ of \emph{vertices} and the set $E$ of \emph{edges} are finite. We denote by $r:E\to V$ the \emph{range map} and by $s:E\to V$ the \emph{source map}. A  graph is \emph{undirected} if for every $e \in E$ there is a $f \in E$ with $s(f) = r(e)$ and $r(f) = s(e)$; it is \emph{directed} otherwise. Elements $e \in E$ with $s(e) = r(e)$ are called \emph{loops}.  A graph \emph{without multiple edges} is a directed graph, where there are no $e,f \in E$, $e \neq f$, such that $s(e) = s(f)$ and $r(e) =r(f)$.  For a finite graph $\Gamma=(V,E)$ with $V=\{1,\ldots,n\}$, its \emph{adjacency matrix} $\epsilon\in M_n(\N_0)$ is defined as $\epsilon_{ij}:=\#\{e\in E\;|\; s(e)=i, r(e)=j\}$. Here $\N_0=\{0,1,2,\ldots\}$.
\textbf{Throughout this article we restrict to finite graphs having no multiple edges.}

If  $\Gamma = (V,E)$ is a directed graph without multiple edges, we denote by $\Gamma^c = (V,E')$ the \emph{complement} of $\Gamma$, where $E' = (V \times V)\backslash E$. Within the category of graphs having no loops, the complement $\Gamma^c$ is defined using $E' = (V \times V)\backslash (E \cup \{(i,i); \ i \in V\})$.

\subsection{Automorphism groups of graphs}

For a finite graph $\Gamma=(V,E)$ without multiple edges, a \emph{graph automorphism} is a bijective map $\sigma: V \to V$ such that $(\sigma(i), \sigma(j)) \in E$ if and only if $ (i,j) \in E$. In other words, $\epsilon_{\sigma(i)\sigma(j)}=1$ if and only if $\epsilon_{ij}=1$. The set of all graph automorphisms of $\Gamma$ forms a group, the \emph{automorphism group} $\Aut(\Gamma)$. We can view $\Aut(\Gamma)$ as a subgroup of the symmetric group $S_n$, if $\Gamma$ has $n$ vertices:
\begin{align*}
\mathrm{Aut}(\Gamma) = \{ \sigma \in S_n\;|\; \sigma \varepsilon = \varepsilon \sigma\}\subset S_n
\end{align*}

\subsection{Graph $C^*$-algebras}\label{SectGraphCStar}

The theory of Graph $C^*$-algebras has its roots in Cuntz and Krieger's work  \cite{CK} in 1980. Nowadays, it forms a well-developed and very active part of the theory of $C^*$-algebras, see \cite{Rae} for an overview or \cite{Eil} for recent developments. For a finite, directed graph $\Gamma=(V,E)$ without multiple edges, the \emph{graph $C^*$-algebra} $C^*(\Gamma)$ is the universal $C^*$-algebra generated by mutually orthogonal projections $p_v$, $v \in V$ and partial isometries $s_e$, $e \in E$ such that 
\begin{itemize}
\item[(i)] $s_e^* s_e = p_{r(e)}$ for all $e \in E$
\item[(ii)] and $p_v = \sum_{e\in E\;:\;s(e)=v} s_e s_e^*$ for every $v \in V$ with $s^{-1}(v) \neq \emptyset$.
\end{itemize}
It follows immediately, that  $s_e^*s_f = 0$ for $e \neq f$ and $\sum_{v\in V} p_v = 1$  hold true in $C^*(\Gamma)$.

\subsection{Compact matrix quantum groups}\label{SectCMQG}

Compact matrix quantum groups were defined by Woronowicz \cite{CMQG1,CMQG2}  in 1987. They form a special class of compact quantum groups, see \cite{Nesh, Tim} for recent books.
A \emph{compact matrix quantum group} $G$ is a pair $(C(G),u)$, where $C(G)$ is a unital (not necessarily commutative) $C^*$-algebra which is generated by $u_{ij}$, $1 \leq i,j \leq n$, the entries of a matrix $u \in M_n(C(G))$. Moreover, the *-homomorphism $\Delta: C(G) \to C(G) \otimes C(G)$, $u_{ij} \mapsto \sum_{k=1}^n u_{ik} \otimes u_{kj}$ must exist, and $u$ and its transpose $u^{t}$ must be invertible. 

\pagebreak

\begin{ex}\label{ExSn}
As an example, consider the \emph{quantum symmetric group} \linebreak$S_n^+ = (C(S_n^+), u)$ as defined by Wang \cite{WanSn} in 1998. It is the compact matrix quantum group given by
\begin{align*}
C(S_n^+) := C^*(u_{ij} \; | \;u_{ij} = u_{ij}^* = u_{ij}^2, \ \sum_{l=1}^n u_{il} = 1 = \sum_{l=1}^n u_{li} \text{ for all $1 \leq i,j \leq n$}).
\end{align*}
One can show that the quotient of $C(S_n^+)$ by the relations that all $u_{ij}$ commute is exactly $C(S_n)$. Moreover, the symmetric group $S_n$ may be viewed as a compact matrix quantum group $S_n=(C(S_n),u)$, where $u_{ij}:S_n\to\C$ are the evaluation maps of the matrix entries. This justifies the name ``quantum symmetric group''.
\end{ex}

If $G=(C(G),u)$ and $H=(C(H), v)$ are compact matrix quantum groups with $u \in M_n(C(G))$ and $v \in M_n(C(H))$, we say that $G$ is a \emph{compact matrix quantum subgroup} of $H$, if there is a surjective *-homomorphism from $C(H)$ to $C(G)$ mapping generators to generators. In this case we write $G \subseteq H$. As an example: $S_n\subset S_n^+$.  The compact matrix quantum groups $G$ and $H$ are equal as compact matrix quantum groups, writing $G =H$, if we have $G \subseteq H$ and $H \subseteq G$. 

\subsection{Actions of quantum groups}

Let $G=(C(G), u)$ be a compact matrix quantum group and let $B$ be a $C^*$-algebra. 
A \emph{left action} of $G$ on $B$ is a unital *-homomorphism $\alpha: B \to C(G) \otimes B$ such that
\begin{itemize}
\item[(i)] $(\Delta \otimes \mathrm{id}) \circ \alpha = (\mathrm{id} \otimes \alpha) \circ \alpha$
\item[(ii)] and $\alpha(B)(C(G) \otimes 1)$ is linearly dense in $C(G) \otimes B$. 
\end{itemize} 
A \emph{right action} is a unital *-homomorphism $\beta: B \to C(G) \otimes B$ with
\begin{itemize}
\item[(i)] $((F \circ\Delta) \otimes \mathrm{id})) \circ \beta = (\mathrm{id} \otimes \beta) \circ \beta$
\item[(ii)] and $\beta(B)(C(G) \otimes 1)$ is linearly dense in $C(G) \otimes B$,
\end{itemize}
where $F$ is the flip map $F: C(G) \otimes C(G) \to C(G) \otimes C(G)$, $a \otimes b \mapsto b \otimes a$. 
Note that in some articles (for instance in \cite{WanSn}), the property (ii) is replaced by 
\begin{itemize}
\item[(ii')] $(\epsilon\otimes\mathrm{id}) \circ \alpha = \mathrm{id}$
\item[(iii')] and there is a dense *-subalgebra of $B$, such that $\alpha$ restricts to a right coaction of the Hopf *-algebra on the *-subalgebra.
\end{itemize}
One can show that (ii') and (iii') are equivalent to (ii), see \cite{Pod}. 

\subsection{Quantum symmetry group of $n$ points}

According to Wang's work \cite{WanSn}, we know that $S_n^+$ (from Example \ref{ExSn}) is the quantum symmetry group of $n$ points in the sense that 
\begin{itemize}
\item[(i)] $S_n^+$ acts from left and right on
\[C^*(p_1,\ldots,p_n\;|\;p_i=p_i^*=p_i^2, \sum_l p_l=1)\]
by $\alpha(p_i):=\sum_{k=1}^n u_{ik}\otimes p_k$ and $\beta(p_i):=\sum_{k=1}^nu_{ki}\otimes p_k$
\item[(ii)] and $S_n^+$ is maximal with these actions, i.e. any other compact matrix quantum group with actions defined as $\alpha$ and $\beta$ is a compact matrix quantum subgroup of $S_n^+$.
\end{itemize}
See also \cite{QISO} for similar questions around quantum symmetries.

\section{Quantum automorphism groups of graphs}

Wang's work in the 1990's was the starting point of the investigations of quantum symmetry phenomena for discrete structures (within Woronowicz's framework). Note that $n$ points may be viewed as the totally disconnected graph on $n$ vertices. A decade later, Banica and Bichon extended Wang's approach to a theory of quantum automorphism groups of finite graphs. In the sequel, we restrict to finite graphs having no multiple edges.

\subsection{Bichon's quantum automorphism group of a graph}

In 2003, Bichon \cite{QBic} defined a quantum automorphism group as follows.

\begin{defn}
Let $\Gamma = (V, E)$ be a finite graph with $n$ vertices $V = \{1, ... , n \}$ and $m$ edges $E = \{ e_1, ... , e_m \}$. The \emph{quantum automorphism group} $\QBic(\Gamma)$ is the compact matrix quantum group $(C(\QBic(\Gamma)), u)$, where $C(\QBic(\Gamma))$ is the universal $C^*$-algebra with generators $u_{ij}$, $1 \leq i,j \leq n$ and relations
\begin{align}
&u_{ij} = u_{ij}^*, \quad u_{ij}u_{ik} = \delta_{jk}u_{ij} ,\quad  u_{ji}u_{ki} = \delta_{jk}u_{ji}, &&1 \leq i,j,k \leq n,\label{QA1}\\ 
&\sum_{l=1}^n u_{il} = 1 = \sum_{l=1}^n u_{li}, &&1 \leq i \leq n,\label{QA2}\\
&u_{s(e_j)i}u_{r(e_j)k} = u_{r(e_j)k}u_{s(e_j)i} = 0, &&e_j \in E, (i,k) \notin E,\label{QA3}\\
&u_{is(e_j)}u_{kr(e_j)} = u_{kr(e_j)}u_{is(e_j)} = 0, &&e_j \in E, (i,k) \notin E,\label{QA4}\\
&u_{s(e_j)s(e_l)}u_{r(e_j)r(e_l)} = u_{r(e_j)r(e_l)}u_{s(e_j)s(e_l)}, &&e_j, e_l \in E.\label{QA5}
\end{align} 
\end{defn}

In the original definition of Bichon, there is actually another relation which is implied by the others:
\begin{align}
\sum_{l=1}^m u_{s(e_l)s(e_j)}u_{r(e_l)r(e_j)} = 1 = \sum_{l=1}^m u_{s(e_j)s(e_l)}u_{r(e_j)r(e_l)}, && e_j \in E \label{QA6}
\end{align}
Indeed, Relations \eqref{QA6} are implied by Relations \eqref{QA2}, \eqref{QA3} and \eqref{QA4}:
\begin{align*}
\sum_{l=1}^m u_{s(e_l)s(e_j)}u_{r(e_l)r(e_j)} = \sum_{i,k =1}^n u_{is(e_j)} u_{kr(e_j)} =\left(\sum_{i=1}^n u_{is(e_j)}\right)\left(\sum_{k=1}^n u_{kr(e_j)}\right)=1
\end{align*} 

\subsection{Banica's quantum automorphism group of a graph}\label{SectBan}

Two years later, Banica \cite{QBan} gave the following definition.

\begin{defn}
Let $\Gamma =(V, E)$ be a finite graph with $n$ vertices and adjacency matrix $\varepsilon \in M_n(\{0,1\})$.  The \emph{quantum automorphism group} $\QBan(\Gamma)$ is the compact matrix quantum group $(C(\QBan(\Gamma)),u)$, where $C(\QBan(\Gamma))$ is the universal $C^*$-algebra with generators $u_{ij}, 1 \leq i,j \leq n$ and Relations \eqref{QA1}, \eqref{QA2} together with
\begin{align}
u \varepsilon = \varepsilon u \label{QA7},
\end{align}
which is nothing but $\sum_ku_{ik}\epsilon_{kj}=\sum_k\epsilon_{ik}u_{kj}$.
\end{defn}

\subsection{Link between the two definitions}

It is easy to see (\cite[Lemma 3.1.1]{Ful} or \cite[Lemma 6.7]{SW16}) that Banica's definition may be expressed as:
\[C(\QBan(\Gamma)) = C^*( u_{ij} \; | \; \textnormal{Relations  \eqref{QA1} -- \eqref{QA4}})\]
We thus have
\begin{align*}
\mathrm{Aut}(\Gamma) \subseteq \QBic(\Gamma) \subseteq \QBan(\Gamma)
\end{align*}
in the sense of compact matrix quantum subgroups, see Section \ref{SectCMQG}. Equality holds, if $C(\QBic(\Gamma))$ and $C(\QBan(\Gamma))$ are commutative.
Moreover, note that (see Example \ref{ExSn}): 
\[C(S_n^+) = C^*(u_{ij} \; | \; \textnormal{Relations  \eqref{QA1} and \eqref{QA2}})\]

\begin{ex}\label{ExQAut}
As an example, let $\Gamma$ be the complete graph (i.e. $E=V\times V$). Then:
\[\Aut(\Gamma)=\QBic(\Gamma)=S_n,\qquad \QBan(\Gamma)=S_n^+\]
For its complement $\Gamma^c$ (i.e. $E=\emptyset$), we have:
\[\Aut(\Gamma^c)=S_n,\qquad \QBic(\Gamma^c)=\QBan(\Gamma^c)=S_n^+\]
\end{ex}

\subsection{Review of the literature on quantum automorphism groups of graphs}\label{SectLit}

 At the moment there are only few articles regarding quantum automorphism groups of graphs. Some results are the following. In \cite{Bic2}, Bichon defined the hyperoctahedral quantum group and showed that this group is the quantum automorphism group of some graph. Banica computed the Poincar\'e series of $\QBan(\Gamma)$ for homogenous graphs with less than eight vertices in \cite{QBan}. Banica, Bichon and Chenevier considered circulant graphs having $p$ vertices for $p$ prime in \cite{Che}. They proved $\QBan(\Gamma) = \mathrm{Aut}(\Gamma)$ if the graph $\Gamma$ does fulfill certain properties. Banica and Bichon investigated  $\QBan(\Gamma)$ for vertex-transitive graphs of order less or equal to eleven in \cite{BanBic}. They also computed $\QBan(\Gamma)$ for the direct product, the Cartesian product and the lexicographic product of specific graphs. Chassaniol also studied the lexicographic product of graphs in \cite{Ca}. In her PhD thesis \cite{Ful}, Fulton studied undirected trees $\Gamma$ such that $\mathrm{Aut}(\Gamma)= \mathbb{Z}_2 \times \mathbb{Z}_2 \times ... \times \mathbb{Z}_2$, where we have $k$ kopies of the cyclic group $\mathbb{Z}_2=\Z/2\Z$. She proved $\mathrm{Aut}(\Gamma) =\QBic(\Gamma)=\QBan(\Gamma)$ for $k=1$ and $\mathrm{Aut}(\Gamma) \neq\QBic(\Gamma)=\QBan(\Gamma)$ for $k\geq 2$. See also \cite{GoBhSk} for links to quantum isometry groups.

\subsection{Comparing with the complement of the graph}\label{SectComplement}

As can be seen from Section \ref{SectLit}, the theory of quantum automorphism groups of graphs is still in its infancy. We now provide some basic results on the link between $\QBic(\Gamma)$ and $\QBic(\Gamma^c)$. Note that while we have
\[\Aut(\Gamma)=\Aut(\Gamma^c)\]
and
\[\QBan(\Gamma)=\QBan(\Gamma^c)\]
 for all graphs $\Gamma$ (using $\epsilon_{\Gamma^c}=A-\epsilon_{\Gamma}$ for the adjacency matrices, with $A\in M_n(\{1\})$ the matrix filled with units, and $uA=A=Au$ by Relation \eqref{QA2}), we may have
 \[\QBic(\Gamma)\neq \QBic(\Gamma^c),\]
 for instance when $\Gamma$ is the complete graph, see Example \ref{ExQAut}.

\begin{lem}\label{same}
If $\QBic(\Gamma) \subset \QBic(\Gamma^c)$, then $\QBic(\Gamma) = \mathrm{Aut}(\Gamma)$. 
\end{lem}
\begin{proof}
Relation \eqref{QA5} in $C(\QBic(\Gamma^c))$ implies that $u_{ik}$ and $u_{jl}$ commute in\linebreak $C(\QBic(\Gamma))$ whenever $(i,j)\notin E$ and $(k,l)\notin E$. Together with Relations \eqref{QA3}, \eqref{QA4} and \eqref{QA5} in $C(\QBic(\Gamma))$ this yields commutativity of all generators.
\end{proof}

\begin{lem}\label{compl}
If $\QBic(\Gamma^c) = \QBan(\Gamma^c)$, then $\QBic(\Gamma) = \mathrm{Aut}(\Gamma)$.
\end{lem}
\begin{proof}
We have $\QBic(\Gamma)\subset \QBan(\Gamma)=\QBan(\Gamma^c)=\QBic(\Gamma^c)$ and apply Lemma \ref{same}.
\end{proof}

The next lemma shows that the quantum automorphism groups of a graph without loops does not change if we add those. 
\begin{lem}\label{withoutloops}
Let $\Gamma =(V,E)$ be a finite graph without loops. Consider $\Gamma' = (V, E')$ with $E' = E \cup \{ (i,i), i \in V \}$. It holds
\begin{itemize}
\item[(i)] $\QBan(\Gamma) = \QBan(\Gamma')$,
\item[(ii)] $\QBic(\Gamma) = \QBic(\Gamma')$.
\end{itemize}
\end{lem}
\begin{proof}
For (i), we use $\epsilon_{\Gamma'}=1+\epsilon_\Gamma$, where $1$ is the identity matrix in $M_n(\{0,1\})$. Thus, $u \varepsilon_{\Gamma} = \varepsilon_{\Gamma} u$ is equivalent to $u \varepsilon_{\Gamma'} = \varepsilon_{\Gamma'} u$. 

For (ii), all we need to check is that  $u_{is(e_j)}u_{ir(e_j)} = u_{ir(e_j)}u_{is(e_j)}$ is fulfilled in $C(\QBic(\Gamma))$ for all $i \in V$, $e_j \in E$, which is true due to Relation \eqref{QA1}.
\end{proof}

\subsection{Quantum automorphism groups on four vertices}\label{SectTabelle}

For a small number of vertices of undirected graphs, a complete classification of $\QBic(\Gamma)$ and $\QBan(\Gamma)$ is possible. For $n\in\{1,2,3\}$, we have $C(S_n^+)=C(S_n)$, hence $\Aut(\Gamma)=\QBic(\Gamma)=\QBan(\Gamma)$. For $n=4$, we now provide a complete table for graphs having no loops. We restrict to undirected graphs in order to keep it simple. We need the following lemma to compute the quantum automorphism groups.

\begin{lem}\label{eqzero}
Let $\Gamma = (V,E)$ be a finite graph with $V=\{ 1, ..., n\}$ and let $e_j \in E$. Let $q\in V$ with $s^{-1}(q)=\emptyset$. For the generators of $C(\QBan(\Gamma))$ it holds
\begin{align*}
u_{q s(e_j)} = 0 = u_{s(e_j)q}. 
\end{align*}
\end{lem}

\begin{proof}
By Relations \eqref{QA2} and \eqref{QA4}, we get 
\begin{align*}
u_{q s(e_j)} =  u_{qs(e_j)}\left(\sum_{i =1}^n u_{ir(e_j)}\right) = \sum_{i=1}^n  u_{qs(e_j)} u_{ir(e_j)}= 0,
\end{align*}
because $(q,i) \notin E$ for all $i \in V$. Likewise, we get $u_{s(e_j)q} = 0$. 
\end{proof}

 In the following, $D_4$ denotes the dihedral group defined as 
\begin{align*}
D_4 := \langle x,y \,|\, x^2=y^2=(xy)^4 = e \rangle,
\end{align*}
$H_2^+$ denotes the hyperoctahedral quantum group defined by Bichon in \cite{Bic2} and $\Z_2$ denotes the cyclic group $\Z/2\Z$.
The quantum group $\widehat{\Z_2 *\Z_2}=(C^*(\Z_2*\Z_2),u)$ is understood as the compact matrix quantum group with matrix 
\[\begin{pmatrix}
p&1-p&0&0\\
1-p&p&0&0\\
0&0&q&1-q\\
0&0&1-q&q
\end{pmatrix}\]
where $C^*(\Z_2*\Z_2)$ is seen as the universal unital $C^*$-algebra generated by two projections $p$ and $q$.
Recall that $\Aut(\Gamma)=\Aut(\Gamma^c)$ and $\QBan(\Gamma)=\QBan(\Gamma^c)$, where $\Gamma^c$ is the complement of $\Gamma$ within the category of graphs having no loops.  Parts of the following table were also computed in \cite{BanBic} and \cite{Bic2}. 

\setlength{\unitlength}{0.5cm}
\newsavebox{\boxgraph}
   \savebox{\boxgraph}
   { \begin{picture}(0.7,0.5)
     \put(-0.5,-0.5){$\bullet$}
     \put(-0.5,0.5){$\bullet$}
     \put(0.5,-0.5){$\bullet$}
     \put(0.5,0.5){$\bullet$}
     \end{picture}}
     
\newsavebox{\lineo}
   \savebox{\lineo}
   { \begin{picture}(0.7,0.5)
     \put(-0.25,0.75){\line(1,0){1}}
     \end{picture}}
    
\newsavebox{\lineu}
   \savebox{\lineu}
   { \begin{picture}(0.7,0.5)
     \put(-0.25, -0.25){\line(1,0){1}}
     \end{picture}}    
     
\newsavebox{\linel}
   \savebox{\linel}
   { \begin{picture}(0.7,0.5)
     \put(-0.35, -0.25){\line(0,1){1}}
     \end{picture}}        
     
\newsavebox{\liner}
   \savebox{\liner}
   { \begin{picture}(0.7,0.5)
     \put(0.7, -0.25){\line(0,1){1}}
     \end{picture}}

\newsavebox{\linelr}
   \savebox{\linelr}
   { \begin{picture}(0.7,0.5)
     \put(0.7, -0.25){\line(-1,1){1}}
     \end{picture}}

\newsavebox{\linerl}
   \savebox{\linerl}
   { \begin{picture}(0.7,0.5)
     \put(-0.25, -0.25){\line(1,1){1}}
     \end{picture}}

\begin{thm}
Let $\Gamma$ be an undirected graph  on four vertices having no loops and no multiple edges. Then:
\end{thm}

\setlength{\tabcolsep}{3mm}
\renewcommand{\arraystretch}{1.4}
\begin{tabular}[t]{c p{0.5cm} p{0.5cm} c c c c}
&$\Gamma$ & $\Gamma^c$ &$\mathrm{Aut}(\Gamma)$ & $\QBic(\Gamma^c)$ & $\QBic(\Gamma)$ & $\QBan(\Gamma)$ \\
\toprule
(1)
&\begin{picture}(0,0) \put(0,0){\usebox{\boxgraph}}\end{picture} 
&\begin{picture}(0,0) \put(0,0){\usebox{\boxgraph}}\put(0,0){\usebox{\linelr}}\put(0,0){\usebox{\linel}}\put(0,0){\usebox{\lineu}}\put(0,0){\usebox{\lineo}}\put(0,0){\usebox{\liner}}\put(0,0){\usebox{\linerl}}\end{picture} 
& $S_4$&$S_4$&$S_4^{+}$& $S_4^{+}$ \\
\midrule
(2)
&\begin{picture}(0,0) \put(0,0){\usebox{\boxgraph}}\put(0,0){\usebox{\lineo}}\end{picture}
& \begin{picture}(0,0) \put(0,0){\usebox{\boxgraph}}\put(0,0){\usebox{\linelr}}\put(0,0){\usebox{\linel}}\put(0,0){\usebox{\lineu}}\put(0,0){\usebox{\lineo}}\put(0,0){\usebox{\liner}}\end{picture}
&$\mathbb{Z}_2 \times \mathbb{Z}_2$&$\mathbb{Z}_2 \times \mathbb{Z}_2$ & $\widehat{\mathbb{Z}_2 * \mathbb{Z}_2}$ & $\widehat{\mathbb{Z}_2 * \mathbb{Z}_2}$\\ 
\midrule
(3)
&\begin{picture}(0,0) \put(0,0){\usebox{\boxgraph}}\put(0,0){\usebox{\lineo}}\put(0,0){\usebox{\linel}}\end{picture}
& \begin{picture}(0,0) \put(0,0){\usebox{\boxgraph}}\put(0,0){\usebox{\linelr}}\put(0,0){\usebox{\linel}}\put(0,0){\usebox{\lineu}}\put(0,0){\usebox{\lineo}}\end{picture}
&$\mathbb{Z}_2$&$\mathbb{Z}_2$ & $\mathbb{Z}_2$ & $\mathbb{Z}_2$\\
\midrule
(4)
&\begin{picture}(0,0) \put(0,0){\usebox{\boxgraph}}\put(0,0){\usebox{\lineo}}\put(0,0){\usebox{\lineu}}\end{picture}
&\begin{picture}(0,0) \put(0,0){\usebox{\boxgraph}}\put(0,0){\usebox{\lineo}}\put(0,0){\usebox{\liner}}\put(0,0){\usebox{\linel}}\put(0,0){\usebox{\lineu}}\end{picture} 
&$D_4$&$D_4$ & $H_2^+$ & $H_2^+$\\
\midrule
(5)
&\begin{picture}(0,0) \put(0,0){\usebox{\boxgraph}}\put(0,0){\usebox{\linel}}\put(0,0){\usebox{\lineo}}\put(0,0){\usebox{\linerl}}\end{picture} 
&\begin{picture}(0,0) \put(0,0){\usebox{\boxgraph}}\put(0,0){\usebox{\linelr}}\put(0,0){\usebox{\linel}}\put(0,0){\usebox{\lineo}}\end{picture} 
&$S_3$&$S_3$ & $S_3$ & $S_3$\\
\midrule
(6)
& \begin{picture}(0,0) \put(0,0){\usebox{\boxgraph}}\put(0,0){\usebox{\lineo}}\put(0,0){\usebox{\linel}}\put(0,0){\usebox{\liner}}\end{picture}
&\begin{picture}(0,0) \put(0,0){\usebox{\boxgraph}}\put(0,0){\usebox{\lineo}}\put(0,0){\usebox{\linel}}\put(0,0){\usebox{\liner}}\end{picture}
&$\mathbb{Z}_2$&$\mathbb{Z}_2$ & $\mathbb{Z}_2$ & $\mathbb{Z}_2$\\
\bottomrule
\end{tabular}


\begin{proof}
For every row of the table, we compute $\QBan(\Gamma)$ and we show $\QBan(\Gamma)=\QBic(\Gamma)$. We then obtain $\QBic(\Gamma^c)$ by using Lemma \ref{compl}. We label the points of the graphs as follows:

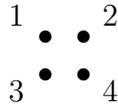
\begin{center}
\begin{picture}(3,3) 
\put(1,1){\usebox{\boxgraph}}
\put(0,2)1
\put(2.5,2)2
\put(0,0)3
\put(2.5,0)4
\end{picture}
\end{center}

\begin{itemize}
\item[(1)] Obvious, see Example \ref{ExQAut}.

\item[(2)] Let $(u_{ij})_{1 \leq i,j \leq 4}$ be the generators of $C(\QBan(\Gamma))$. Lemma \ref{eqzero} yields 
\begin{align*}
u_{31}= u_{32}= u_{41}= u_{42}= u_{13}= u_{23} =u_{14} =u_{24}= 0.
\end{align*}
With Relations \eqref{QA2} we deduce
\begin{align*}
u=\begin{pmatrix} u_{11} & 1 - u_{11} & 0 & 0\\ 1 - u_{11} & u_{11} & 0&0 \\ 0& 0& u_{33}& 1- u_{33}\\ 0& 0& 1- u_{33}& u_{33}\end{pmatrix}.
\end{align*}
Thus
\begin{align*}
\QBan(\Gamma) = \widehat{\mathbb{Z}_2 * \mathbb{Z}_2}.
\end{align*}
Since $u_{ij}u_{kl} = u_{kl}u_{ij}$ holds for $(i,k), (j,l) \in \{(1,2), (2,1)\}$ in \linebreak$C(\QBan(\Gamma))$, we get $\QBic(\Gamma)=\QBan(\Gamma)$.

\item[(3)]
Lemma \ref{eqzero} yields 
\begin{align*}
u_{14} = u_{24} = u_{34} = u_{41} = u_{42} = u_{43} = 0.
\end{align*}
This implies
\begin{align*}
\QBan(\Gamma) \subseteq S_3^+ = S_3, 
\end{align*}
thus $\QBan(\Gamma)$ is commutative and hence $\QBan(\Gamma)=\QBic(\Gamma)=\Aut(\Gamma)=\Z_2$.

\item[(4)]Let $\Delta$ and $\Delta'$ be the comultiplication maps of $\QBan(\Gamma)$ and $H_2^+$, respectively. We first show that these two quantum groups coincide as compact quantum groups, i.e. there is a $^*$-isomorphism 
\[\phi: C(H_2^+)\to C(\QBan(\Gamma))\]
such that $\Delta'\circ\phi=(\phi\otimes\phi)\circ\Delta$.\\


\noindent \emph{Step 1: The map $\phi$ exists and we have $\Delta'\circ\phi=(\phi\otimes\phi)\circ\Delta$.}\\
From $\varepsilon u = u \varepsilon$ we get
\begin{align*}
u = \begin{pmatrix} u_{11} & u_{12} & u_{13} & u_{14}\\ u_{12} & u_{11} &u_{14} & u_{13}\\ u_{31}& u_{32} & u_{33} & u_{34}\\ u_{32}& u_{31}& u_{34}& u_{33} \end{pmatrix}.
\end{align*}
Define $v_{11} := u_{11} - u_{12}$, $v_{12} := u_{13} - u_{14}$, $v_{21} := u_{31} - u_{32}$ and $v_{22} := u_{33} - u_{34}$. One can compute that
 $v_{ij}$, $i,j =1,2$ fulfill the relations of $C(H_2^+)$ and with the universal property we get a *-homomorphism $\varphi: C(H_2^+) \to C(\QBan(\Gamma))$. Since $\Delta' \circ \varphi = (\varphi \otimes \varphi) \circ \Delta$ also holds, we get that $\QBan(\Gamma)$ is a quantum subgroup of $H_2^+$.\\

\noindent \emph{Step 2: The map $\phi$ is a $^*$-isomorphism.}\\
Let $(v_{ij})_{i,j=1,2}$ be the generators of $C(H_2^+)$. Define 
\begin{align*}
u_{11} &:= u_{22}:= \frac{v_{11}^2 + v_{11}}{2}, \qquad u_{12} := u_{21}:= \frac{v_{11}^2 - v_{11}}{2},\\
u_{13} &:= u_{24}:= \frac{v_{12}^2 + v_{12}}{2}, \qquad u_{14} := u_{23}:= \frac{v_{12}^2 - v_{12}}{2},\\
u_{31} &:= u_{42}:= \frac{v_{21}^2 + v_{21}}{2}, \qquad u_{41} := u_{32}:= \frac{v_{21}^2 - v_{21}}{2},\\
u_{33} &:= u_{44}:= \frac{v_{22}^2 + v_{22}}{2}, \qquad u_{34} := u_{43}:= \frac{v_{22}^2 - v_{22}}{2}.
\end{align*}
One can show that the $(u_{ij})_{1 \leq i,j \leq 4}$ fulfill the relations of $C(\QBan(\Gamma))$. The universal property now gives us a *-homomorphism $\varphi' : C(\QBan(\Gamma)) \to C(H_2^+)$ and $\phi'$ is the inverse of $\phi$ and vice versa.\\ 

\noindent \emph{Step 3: We have $\QBan(\Gamma) = \QBic(\Gamma)$.}\\
We have seen in Step 1, that
\begin{align*}
u_{11} = u_{22},\quad u_{12} &= u_{21}, \quad u_{13} = u_{24}, \quad u_{14} = u_{23}, \\
u_{31} = u_{42}, \quad u_{32} &= u_{41}, \quad u_{33} = u_{44}, \quad u_{34} = u_{43}
\intertext{ and therefore we get}
u_{ij} u_{kl} &= u_{kl}^2 = u_{kl} u_{ij} 
\end{align*}  
for all $(i,k), (j,l) \in E$. Thus $\QBan(\Gamma) = \QBic(\Gamma)$. 


\item[(5)] We conclude as in (3).

\item[(6)] Some direct computations using $\varepsilon u = u \varepsilon$ and Relations \eqref{QA2} show
\begin{align*}
u = \begin{pmatrix} u_{33}& 1-u_{33}& 0&0\\ 1-u_{33} & u_{33}& 0& 0\\ 0& 0& u_{33}& 1-u_{33}\\ 0& 0& 1-u_{33}& u_{33}\end{pmatrix}.
\end{align*}
Thus $\QBan(\Gamma)$ is commutative.
\end{itemize}
\end{proof}

\pagebreak

\section{Proof of the main result}

We now prove the main result of this article (see Section \ref{SectMain}) for a finite graph $\Gamma$ with vertices $V=\{1,\ldots,n\}$ and edges $E=\{e_1,\ldots,e_m\}$ having no multiple edges.

\begin{rem}\label{RemQSym}
We define the quantum symmetry group $\textnormal{QSym}(C^*(\Gamma))$ of $C^*(\Gamma)$ to be the maximal compact matrix quantum group $G$ acting on $C^*(\Gamma)$ by $\alpha: C^*(\Gamma) \to C(G)\otimes C^*(\Gamma)$ and $\beta: C^*(\Gamma) \to C(G)\otimes C^*(\Gamma)$ as defined in the statement of our main theorem. We thus have to show that $\QBan(\Gamma)$ acts on $C^*(\Gamma)$ via $\alpha$ and $\beta$ (see Sections \ref{existence} and \ref{Step42}) and that it is maximal with these actions (see Section \ref{maximal}).
\end{rem}

\subsection{Existence of the maps $\alpha$ and $\beta$}\label{existence}

In order to prove that 
\begin{align*}
\alpha: C^*(\Gamma) &\to C(\QBan(\Gamma))\otimes C^*(\Gamma) \\
p_i &\mapsto p_i':=\sum_{k = 1}^n u_{ik} \otimes p_k, &&1 \leq i \leq n\\
s_{e_j} &\mapsto s_{e_j}':=\sum_{l=1}^m  u_{s(e_j)s(e_l)}u_{r(e_j)r(e_l)} \otimes s_{e_l}, &&1 \leq j \leq m
\end{align*} 
defines a $*$-homomorphism, all we have to show is that the relations of $C^*(\Gamma)$ hold for $p_i'$ and $s_{e_j}'$. We may then use the universal property of $C^*(\Gamma)$. The proof for the existence of $\beta$ is analogous.

\subsubsection{The $p_i'$ are mutually orthogonal projections}

Obviously, $p_i' = (p_i')^*$ holds. Moreover, using $p_kp_l=\delta_{kl}p_k$ and Relations \eqref{QA1}, we have
\[p_i'p_j'=\sum_{k,l=1}^n u_{ik}u_{jl}\otimes p_kp_l=\sum_{k=1}^n u_{ik}u_{jk}\otimes p_k=\delta_{ij}p_i'.\]

\subsubsection{The $s_{e_j}'$ are partial isometries with $(s_{e_j}')^* s_{e_j}' = p_{r(e_j)}'$}

Using $s_{e_l}^*s_{ei}=\delta_{il}p_{r(e_i)}$ (see Section \ref{SectGraphCStar}) and Relations \eqref{QA1}, we have
\begin{align*}
(s_{e_j}')^* s_{e_j}'
&=\sum_{l,i = 1}^m u_{r(e_j)r(e_l)}u_{s(e_j)s(e_l)}u_{s(e_j)s(e_i)}u_{r(e_j)r(e_i)} \otimes s_{e_l}^*s_{e_i}\\
&=\sum_{i = 1}^m u_{r(e_j)r(e_i)}u_{s(e_j)s(e_i)}u_{r(e_j)r(e_i)} \otimes p_{r(e_i)}.
\end{align*}
By Relations \eqref{QA3} we have $u_{r(e_j)j'}u_{s(e_j)i'} u_{r(e_j)j'} = 0$ for $(i',j') \notin E$. This yields
\begin{align*} 
\sum_{i = 1}^m  u_{r(e_j)r(e_i)}u_{s(e_j)s(e_i)}u_{r(e_j)r(e_i)} \otimes p_{r(e_i)}= \sum_{i', j' =1}^n u_{r(e_j)j'}u_{s(e_j)i'}u_{r(e_j)j'} \otimes p_{j'}.
\end{align*}
Using Relations \eqref{QA2}, we obtain $\sum_{i=1}^nu_{s(e_j)i'}=1$ and thus
\begin{align*}
(s_{e_j}')^* s_{e_j}' &=\sum_{i', j' =1}^n u_{r(e_j)j'}u_{s(e_j)i'}u_{r(e_j)j'} \otimes p_{j'}=\sum_{j' = 1}^n u_{r(e_j)j'} \otimes p_{j'}=p_{r(e_j)}'. 
\end{align*}

\subsubsection{We have $\sum_{j:\;s(e_j)= v} s_{e_j}' (s_{e_j}')^* = p_v'$ for $s^{-1}(v)\neq\emptyset$}
\label{Step413}

Using Relations \eqref{QA1}, we get for $v\in V$ with $s^{-1}(v)\neq\emptyset$:
\begin{align*}
\sum_{\substack{j\in\{1,\ldots,m\}\\s(e_j)= v}} s_{e_j}' (s_{e_j}')^* 
&=\sum_{\substack{j\in\{1,\ldots,m\}\\s(e_j)= v}}\sum_{i,l=1}^m u_{vs(e_l)}u_{r(e_j)r(e_l)}u_{r(e_j)r(e_i)}u_{vs(e_i)}\otimes s_{e_l}s_{e_i}^*\\
&= \sum_{l=1}^m \sum_{\substack{i\in\{1,\ldots,m\}\\r(e_i)=r(e_l)}} u_{vs(e_l)}\left(\sum_{\substack{j\in\{1,\ldots,m\}\\s(e_j)= v}}u_{r(e_j)r(e_l)}\right)u_{vs(e_i)}\otimes s_{e_l}s_{e_i}^*\\
\end{align*}
Now,
\[\sum_{\substack{j\in\{1,\ldots,m\}\\s(e_j)= v}}u_{r(e_j)r(e_l)}=\sum_{\substack{q\in V\\(v,q)\in E}}u_{qr(e_l)}\]
 and for  $q\in V$ with  $(v,q) \notin E$ we have $u_{vs(e_l)}u_{qr(e_l)}=0$ by Relations \eqref{QA4}. 
Thus, for any $l\in\{1,\ldots,m\}$, we have using Relations \eqref{QA2}
\[u_{vs(e_l)}\sum_{\substack{q\in V\\(v,q)\in E}}u_{qr(e_l)}=u_{vs(e_l)}\sum_{q\in V}u_{qr(e_l)}=u_{vs(e_l)}\]
and hence:
\begin{align*}
\sum_{\substack{j\in\{1,\ldots,m\}\\s(e_j)= v}} s_{e_j}' (s_{e_j}')^*  &= \sum_{l=1}^m \sum_{\substack{i\in\{1,\ldots,m\}\\r(e_i)=r(e_l)}}  u_{vs(e_l)}u_{vs(e_i)}\otimes s_{e_l}s_{e_i}^*
\end{align*}
Since $\Gamma$ has no multiple edges by assumption, $r(e_i)=r(e_l)$ and $s(e_i)=s(e_l)$ implies $e_i=e_l$. We thus infer using Relations \eqref{QA1}:
\begin{align*}
\sum_{\substack{j\in\{1,\ldots,m\}\\s(e_j)= v}} s_{e_j}' (s_{e_j}')^*&= \sum_{l=1}^m  u_{vs(e_l)}\otimes s_{e_l}s_{e_l}^*
\end{align*}
Now, for $V':=\{q\in V\;|\; s^{-1}(q)\neq \emptyset\}$, we have, using the relations in $C^*(\Gamma)$:
\begin{align*}
\sum_{l=1}^m  u_{vs(e_l)}\otimes s_{e_l}s_{e_l}^*=\sum_{q \in V'} \sum_{\substack{l\in\{1,\ldots,m\}\\s(e_l) =q}} u_{vq} \otimes s_{e_l}s_{e_l}^*=\sum_{q\in V'}u_{vq}\otimes p_q
\end{align*}

Since we know that $u_{vq} = 0$ for $q \notin V'$ by Lemma \ref{eqzero}, we finally conclude
\begin{align*}
\sum_{\substack{j\in\{1,\ldots,m\}\\s(e_j)= v}} s_{e_j}' (s_{e_j}')^* =\sum_{q =1}^n u_{vq} \otimes p_q =p_v'.
\end{align*}
This settles the existence of $\alpha$.

\subsection{The map $\alpha$ is a left action and $\beta$ is a right action}
\label{Step42}

We only prove this claim for $\alpha$, the proof for $\beta$ being analogous.

\subsubsection{$(\Delta \otimes \mathrm{id}) \circ \alpha=(\mathrm{id} \otimes \alpha)\circ \alpha$ holds and $\alpha$ is unital}

Using Relations \eqref{QA3}, this is straightforward to check.

It remains to show that
\[\mathcal S:=\mathrm{span}\; \alpha(C^*(\Gamma))(C(\QBan(\Gamma)) \otimes 1)\]
is dense in $C(\QBan(\Gamma)) \otimes C^*(\Gamma)$, which we will do in the sequel.

\subsubsection{The elements $1\otimes p_l,1\otimes s_{e_l}$ and $1\otimes s_{e_l}^*$ are in $\mathcal S$}
 Using Relations \eqref{QA1} and \eqref{QA2} we infer:
\begin{align*}
\mathcal S\ni\sum_{i=1}^n \alpha(p_i)(u_{il} \otimes 1) = \sum_{i=1}^n \sum_{j=1}^n u_{ij} u_{il} \otimes p_j = \sum_{i=1}^n u_{il} \otimes p_l = 1 \otimes p_l
\end{align*}
Moreover, for $e_l \in E$ we get, using Relations \eqref{QA1} and $V':=\{v\in V\;|\; s^{-1}(v)\neq\emptyset\}$:
\begin{align*}
\sum_{v \in V'} &\sum_{\substack{j\in\{1,\ldots,m\}\\s(e_j)= v}} \alpha(s_{e_j})(u_{r(e_j)r(e_l)}u_{vs(e_l)} \otimes 1)\\
&=\sum_{v \in V'} \sum_{\substack{j\in\{1,\ldots,m\}\\s(e_j)= v}} \left(\sum_{k=1}^m u_{vs(e_k)}u_{r(e_j)r(e_k)}u_{r(e_j)r(e_l)}u_{vs(e_l)} \otimes s_{e_k}\right)\\
&=\sum_{v \in V'} \left(\sum_{\substack{k\in\{1,\ldots,m\}\\r(e_k)=r(e_l)}} u_{vs(e_k)}\left(\sum_{\substack{j\in\{1,\ldots,m\}\\s(e_j)= v}} u_{r(e_j)r(e_l)}\right)u_{vs(e_l)} \otimes s_{e_k}\right)
\end{align*}
We proceed similar to Step \ref{Step413}.
By Relations \eqref{QA4}, we know $u_{qr(e_l)}u_{vs(e_l)} =0$ for $(v,q) \notin E$. Thus, by Relations \eqref{QA1} and \eqref{QA2} and using that $\Gamma$ has no multiple edges, we obtain:
\allowdisplaybreaks
\begin{align*}
\sum_{v \in V'} &\sum_{\substack{j\in\{1,\ldots,m\}\\s(e_j)= v}} \alpha(s_{e_j})(u_{r(e_j)r(e_l)}u_{vs(e_l)} \otimes 1)\\
&=\sum_{v \in V'} \left(\sum_{\substack{k\in\{1,\ldots,m\}\\r(e_k)=r(e_l)}} u_{vs(e_k)}\left(\sum_{q=1}^n u_{qr(e_l)}\right)u_{vs(e_l)} \otimes s_{e_k}\right)\\
&=\sum_{v \in V'} \left(\sum_{\substack{k\in\{1,\ldots,m\}\\r(e_k)=r(e_l)}} u_{vs(e_k)}u_{vs(e_l)} \otimes s_{e_k}\right)\\
&=\sum_{v \in V'} u_{vs(e_l)} \otimes s_{e_l}
\end{align*}
Finally, Lemma \ref{eqzero} yields $u_{vs(e_l)}=0$ for $v\notin V'$. Hence, using Relations \eqref{QA2}:
\begin{align*}
\mathcal S\ni\sum_{v \in V'} \sum_{\substack{j\in\{1,\ldots,m\}\\s(e_j)= v}} \alpha(s_{e_j})(u_{r(e_j)r(e_l)}u_{s(e_j)s(e_l)} \otimes 1)
=\sum_{i=1}^n u_{is(e_l)} \otimes s_{e_l}
= 1 \otimes s_{e_l}
\end{align*}

Define $V'' := \{v \in V \;|\; r^{-1}(v)\neq\emptyset\}$. Similar to the computations above, we get  
\begin{align*}
\mathcal S\ni\sum_{v \in V''} \sum_{\substack{j\in\{1,\ldots,m\}\\s(e_j)= v}} \alpha(s_{e_j}^*)(u_{s(e_j)s(e_l)}u_{r(e_j)r(e_l)} \otimes 1) = 1 \otimes s_{e_l}^*.
\end{align*}

\subsubsection{If $1\otimes x,1\otimes y\in \mathcal S$, then also $1\otimes xy\in\mathcal S$}
The remainder of the proof of Step \ref{Step42} consists in general facts for actions of compact matrix quantum groups.

We may write $1\otimes x\in\mathcal S$ and $1\otimes y\in\mathcal S$ as
\begin{align*}
1 \otimes x = \sum_{i=1}^l \alpha(z_i)(w_i \otimes 1), \qquad 1 \otimes y = \sum_{j=1}^k \alpha(t_j)(v_j \otimes 1)
\end{align*}    
for some $z_i,t_j \in C^*(\Gamma)$ and $w_i, v_j \in C(\QBan(\Gamma))$. Therefore
\begin{align*}
1 \otimes xy &= \sum_{i=1}^l \alpha(z_i)(w_i \otimes 1)(1 \otimes y)\\
&= \sum_{i=1}^l \alpha(z_i)(1 \otimes y) (w_i \otimes 1)\\
&= \sum_{i=1}^l \sum_{j=1}^k \alpha(z_i t_j) (v_j w_i \otimes 1)\in\mathcal S
\end{align*}

\subsubsection{$\mathcal S$ is dense in $C(\QBan(\Gamma)) \otimes C^*(\Gamma)$}

Summarizing, we get that $1 \otimes w \in \mathcal S$ for all monomials $w$ in $p_i, s_{e_j}, s_{e_j}^*$, $1 \leq i \leq n$, $1 \leq j \leq m$. Since $\alpha$ is unital, we also have:
\[C(\QBan(\Gamma)) \otimes 1 \subseteq \alpha(C^*(\Gamma))(C(\QBan(\Gamma))\otimes 1)\subseteq \mathcal S\]
We conclude that $\mathcal S$ is dense in $C(\QBan(\Gamma)) \otimes C^*(\Gamma)$, which settles Step \ref{Step42}.

\subsection{The quantum group $\QBan(\Gamma)$ acts maximally on $C^*(\Gamma)$}\label{maximal}

For proving the maximality, let $G= (C(G), u)$ be another compact matrix quantum group acting on $C^*(\Gamma)$ by $\alpha': C^*(\Gamma) \to C(G) \otimes C^*(\Gamma)$ and $\beta': C^*(\Gamma) \to  C(G) \otimes C^*(\Gamma)$ in the way $\QBan(\Gamma)$ acts on $C^*(\Gamma)$ via  $\alpha$ and $\beta$. We want to show that there is a *-homomorphism $C(\QBan(\Gamma)) \to C(G)$ sending generators to generators. Thus, we need to compute that the generators $u_{ij}$ of $C(G)$ fulfill the relations of $C(\QBan(\Gamma))$. 

\subsubsection{The Relations \eqref{QA1} hold in $C(G)$}

The equation 
\begin{align*}
\sum_{k=1}^n u_{ik} \otimes p_k = \alpha'(p_i) = \alpha'(p_i)^* = \sum_{k =1}^n u_{ik}^* \otimes p_k
\end{align*}
yields $u_{ij} = u_{ij}^*$ after multiplying from the left with $1\otimes p_j$. We also have 
\begin{align*}
\sum_{i=1}^n u_{ji} u_{ki} \otimes p_i = \sum_{i,l=1}^n u_{ji} u_{kl} \otimes p_i p_l= \alpha'(p_j) \alpha'(p_k)= \delta_{jk}\alpha'(p_j)= \sum_{i=1}^n \delta_{jk} u_{ji} \otimes p_i
\end{align*}
from which we infer $u_{ji}u_{ki}=\delta_{jk}u_{ji}$. Using $\beta'$, we also obtain $u_{ij}u_{ik}=\delta_{jk}u_{ij}$.

\subsubsection{The Relations \eqref{QA2} hold in $C(G)$}

From
\begin{align*}
\sum_{k =1}^n 1 \otimes p_k = 1 \otimes 1 = \alpha'(1) = \sum_{i=1}^n \alpha'(p_i) = \sum_{k=1}^n \left(\sum_{i=1}^n u_{ik} \right) \otimes p_k
\end{align*}
we deduce $\sum_{i=1}^n u_{ik}=1$, and likewise $\sum_{i=1}^nu_{ki}=1$ using $\beta'$.

\subsubsection{The Relations \eqref{QA3} hold in $C(G)$}

Using $s_{e_l}^*s_{e_t}=\delta_{lt}p_{r(e_l)}$ (see Section \ref{SectGraphCStar}) and Relations \eqref{QA1} in $C(G)$, we obtain for any $j$:
\allowdisplaybreaks
\begin{align*}
\sum_{q=1}^n u_{r(e_j)q} \otimes p_q
&=\alpha'(p_{r(e_j)})\\
&=\alpha'(s_{e_j}^*s_{e_j})\\
&= \sum_{l,t=1}^m u_{r(e_j)r(e_l)} u_{s(e_j)s(e_l)}u_{s(e_j)s(e_t)}u_{r(e_j)r(e_t)} \otimes s_{e_l}^*s_{e_t}\\
&=\sum_{l=1}^m u_{r(e_j)r(e_l)}u_{s(e_j)s(e_l)}u_{r(e_j)r(e_l)} \otimes p_{r(e_l)}
\end{align*}
Multiplication with $1\otimes p_k$ yields:
\begin{align*}
u_{r(e_j)k}=\sum_{\substack{l\in\{1,\ldots,m\}\\r(e_l)=k}} u_{r(e_j)k}u_{s(e_j)s(e_l)}u_{r(e_j)k} 
\end{align*}
If $r^{-1}(k)=\emptyset$, then $u_{r(e_j)k}=0$ and hence $u_{s(e_j)i}u_{r(e_j)k}=u_{r(e_j)k}u_{s(e_j)i}=0$ for all $i\in V$.

Otherwise, if $r^{-1}(k)\neq \emptyset$, we use Relations \eqref{QA1} and \eqref{QA2} in $C(G)$ and get
\begin{align*}
\sum_{\substack{l\in\{1,\ldots,m\}\\r(e_l) =k}} u_{r(e_j)k}u_{s(e_j)s(e_l)}u_{r(e_j)k}  = u_{r(e_j)k}= u_{r(e_j)k}^2 
=\sum_{i=1}^n u_{r(e_j)k} u_{s(e_j)i} u_{r(e_j)k}
\end{align*}
 and therefore 
 \begin{align*}
 \sum_{\substack{i\in V\\(i,k)\notin E}}  u_{r(e_j)k} u_{s(e_j)i} u_{r(e_j)k} =0. 
 \end{align*}
 Since 
 \begin{align*}
 u_{r(e_j)k} u_{s(e_j)i} u_{r(e_j)k} =(u_{s(e_j)i} u_{r(e_j)k})^*u_{s(e_j)i} u_{r(e_j)k}
 \end{align*} 
 holds, the above is a vanishing sum of positive elements -- and hence each summand vanishes. This yields $u_{s(e_j)i} u_{r(e_j)k}= 0$ for all $(i,k) \notin E$.
 
 \subsubsection{The Relations \eqref{QA4} hold in $C(G)$}
 
 The argument is analogous to the one for proving Relations \eqref{QA3} when replacing $\alpha'$ by $\beta'$.\\
 
The proof of the main theorem is complete.

\begin{rem}\label{RemBicInsp}
Let $\Gamma$ be a finite graph with $n$ vertices $V=\{ 1, ..., n\}$ and $m$ edges $E = \{ e_1, ..., e_m\}$. In \cite{QBic}, Bichon showed that $\QBic(\Gamma)$ is the quantum symmetry group of $\Gamma$ in his sense, where 
\begin{align*}
&\beta_V: C(V) \to C(\QBic(\Gamma))\otimes C(V) , g_i \mapsto \sum_{k=1}^n  u_{ki}\otimes g_k,\\
&\beta_E: C(E) \to C(\QBic(\Gamma))\otimes C(E) , f_j \mapsto \sum_{l=1}^m u_{s(e_l)s(e_j)}u_{r(e_l)r(e_j)}\otimes f_l,
\end{align*}
define actions of $\QBic(\Gamma)$ on $C(V)$ and $C(E)$, respectively. Those actions inspired us, how an action of a compact matrix quantum group on $C^*(\Gamma)$ should look like. However, note that edges in the commutative $C^*$-algebra $C(E)$ of continuous functions on $E$ are represented as projections unlike in the case of $C^*(\Gamma)$. Therefore, the quantum symmetry group of $C^*(\Gamma)$ is $\QBan(\Gamma)$ rather than $\QBic(\Gamma)$. On the other hand, if we consider the quotient of $C^*(\Gamma)$ by the relations $s_e=s_e^*$, its quantum symmetry group is $\QBic(\Gamma)$. Indeed, selfadjointness of $s_e$ yields
\begin{align*}
\sum_{l=1}^m  u_{s(e_j)s(e_l)}u_{r(e_j)r(e_l)} \otimes s_{e_l} = \alpha(s_{e_j}) = \alpha(s_{e_j})^* = \sum_{l=1}^m  u_{r(e_j)r(e_l)}u_{s(e_j)s(e_l)} \otimes s_{e_l},
\end{align*}
from which we obtain Relations \eqref{QA5} by multiplication with $(1 \otimes s_{e_i}^*)$ from the left.
\end{rem}

\bibliographystyle{plain}
\bibliography{QSym}

\end{document}